\documentclass{amsart}
\usepackage{amsfonts}
\usepackage{amssymb}
\usepackage{times}
\usepackage{xcolor}
\usepackage{amsmath}
\usepackage{amsthm}
\usepackage{comment}
\usepackage{xcolor}
\usepackage{graphicx}
\usepackage{mathtools}
\usepackage{commath}

\makeatletter
\def\Ddots{\mathinner{\mkern1mu\raise\p@
\vbox{\kern7\p@\hbox{.}}\mkern2mu
\raise4\p@\hbox{.}\mkern2mu\raise7\p@\hbox{.}\mkern1mu}}
\makeatother

\newtheorem{theorem}{Theorem}[section]
\newtheorem{corollary}{Corollary}[section]
\newtheorem{lemma}{Lemma}[section]

\newtheorem{proposition}{Proposition}[section]
\newtheorem{definition}{Definition}[section]
\newtheorem{example}{Example}[section]

\begin{document}

\title{On the Additivity of $n$-Multiplicative Isomorphisms, Derivations, and Related Maps in Axial Algebras}

\thanks{The first author was supported by FAPESP number 2022/14579-0. The second author was partially supported by FAPESP number 2018/23690-6}

\author[Daniel Eiti Nishida Kawai and Henrique Guzzo Jr.]{Daniel Eiti Nishida Kawai and Henrique Guzzo Jr.}
\address{Daniel Eiti Nishida Kawai and Henrique Guzzo Jr., S\~{a}o Paulo University, 
Rua Mat\~{a}o, 1010, 
CEP 05508-090, S\~{a}o Paulo, Brazil\\
{\em E-mail}: {\tt daniel.kawai@usp.br, guzzo@ime.usp.br}}{}

\author[Bruno Leonardo Macedo Ferreira]{Bruno Leonardo Macedo Ferreira}
\address{Bruno Leonardo Macedo Ferreira, Federal University of Technology-Paran\'{a}-Brazil and Federal University of ABC, 
Avenida Professora Laura Pacheco Bastos, 800, 
85053-510, Guarapuava, Brazil and Av. dos Estados, 5001 - Bangú, Santo André - SP, 09280-560\\\linebreak
{\em E-mail}: {\tt brunoferreira@utfpr.edu.br } {\em  or } {\tt brunolmfalg@gmail.com}}

\begin{abstract}
In this paper, we demonstrate that several classes of functions, specifically $n$-multiplicative isomorphisms, derivations, elementary maps, and Jordan elementary maps on a class of algebras that includes Jordan algebras with idempotents, $\mathcal{J}(\alpha)$-axial algebras and $\mathcal{M}(\alpha, \beta)$-axial algebras, are additive under appropriate conditions, which may be referred to as Martindale-type conditions. Furthermore, we answer the question left open in the recent article titled ``Multiplicative isomorphisms and derivations on axial algebra."

\vspace*{.1cm}


\vspace*{.1cm}

\end{abstract}

\maketitle

\section{Introduction}

The study of the additivity of multiplicative maps is rooted in the broader investigation of the structure and behavior of algebraic systems, particularly rings and algebras. The fundamental motivation lies in understanding when multiplicative isomorphisms, which naturally preserve the multiplicative structure of these systems, also preserve their additive structure. This question is not merely of academic interest; it has profound implications, particularly in the classification of algebraic structures, the development of functional analysis, and the stability of algebraic invariants. For instance, understanding additivity is crucial in determining when ring isomorphisms can be extended to module isomorphisms, which has direct consequences for the representation theory of algebras. Additionally, in the context of operator algebras, ensuring that multiplicative maps are also additive can be essential in preserving the norm structure and continuity, which are key in applications to quantum mechanics and other areas of mathematical physics. These implications highlight the significance of the problem, as they impact both theoretical advancements and practical applications in various branches of mathematics.

One key motivation comes from the desire to generalize and deepen existing results in ring theory, as seen in Martindale III's 1969 work \cite{18}. He aimed to extend Rickart's theorem \cite{21} on the conditions under which a multiplicative isomorphism is also additive, while simultaneously removing certain minimality conditions that were previously assumed. The removal of these conditions is significant because it broadens the applicability of the results, making them relevant to a wider class of rings. By doing so, Martindale III was able to offer a more comprehensive understanding of the relationship between multiplicative and additive structures in algebraic systems, thus contributing to the foundational knowledge in ring theory.

The study of the additivity of functions, particularly in the context of ring theory, holds significant importance, as evidenced by Johnson's exploration of rings with unique addition, see \cite{17}. The concept of unique addition is fundamental because it establishes a clear and definitive structure for the ring in question, where the addition operation cannot be altered without fundamentally changing the ring itself. This uniqueness has far-reaching implications for the stability and classification of algebraic structures.

When a ring  $R$ possesses a unique addition, any bijective (1-1) multiplicative mapping from $ R$  onto another ring  $ S $ must also be additive. This property is crucial because it guarantees that the multiplicative and additive structures of the ring are inherently linked, preventing the existence of alternative additive operations that could disrupt the ring's algebraic integrity. In contrast, if a ring does not have a unique addition, it becomes possible to define a new addition operation that transforms the ring into a different algebraic entity, even while preserving the multiplicative structure. This potential for altering the ring’s structure through non-additive multiplicative mappings highlights the delicate balance between multiplication and addition in algebraic systems.

Therefore, the importance of studying additivity lies in its role in preserving the fundamental properties of algebraic structures. In applications, this preservation ensures that mathematical models remain consistent and reliable. For example, in functional analysis and operator algebras, the additivity of mappings is crucial for maintaining the continuity and norm-preserving properties of linear operators, which are essential in various areas such as quantum mechanics. Moreover, the unique addition property can simplify the classification of rings by ruling out the existence of alternative, non-equivalent additive structures, thereby streamlining the study of ring isomorphisms and module theory.

Furthermore, Rickart's earlier work in 1948 explored the automatic additivity of certain mappings under specific conditions. His investigation into Boolean rings and rings with minimal right ideals demonstrates the importance of these structures in understanding the behavior of multiplicative mappings. The fact that additivity can be an automatic consequence of certain structural properties of the ring suggests a deep interplay between the ring's architecture and the behavior of mappings defined on it. This insight is crucial because it hints at underlying algebraic principles that govern the interaction between different operations within these structures.

The study of additivity is also motivated by its implications for the theory of linear operators on Banach spaces, as seen in Eidelheit's work from 1940. Eidelheit's exploration of isomorphisms between rings of linear operators highlights the importance of understanding when these isomorphisms preserve not only the multiplicative structure but also the additive structure of the operators. Since linear operators are central to functional analysis and related fields, ensuring that their algebraic isomorphisms respect both addition and multiplication is essential for maintaining the integrity of the mathematical models used in these disciplines.

The historical context provided by Mazur’s work \cite{Sur-les-anneaux-linéaries-1938} on linear rings further emphasizes the significance of these studies. The Gelfand-Mazur theorem, which underpins much of the theory of commutative Banach algebras, illustrates how foundational results in algebra can have far-reaching consequences. Mazur's announcement of this theorem in 1938 and its subsequent proof by Gelfand demonstrate the importance of understanding the basic algebraic structures and the mappings that preserve them. The study of additivity in multiplicative maps can be seen as a continuation of this tradition, seeking to uncover the fundamental properties that govern algebraic systems.

Concisely, Martindale's III question became an active research area in associative ring theory, and his pioneering brought focus to the study of when a multiplicative isomorphism is additive. Along this line, other types of functions have been studied also. For example, Daif in \cite{4} studied this question for multiplicative derivation of a ring, proving that a multiplicative derivation with Martindale conditions is additive. Also, Li and Lu in \cite{elementarie} studied the question of additivity of elementary functions, that were introduced in the papers \cite{BresarandSemrl1999, Bresaretal2000, MolnarandSemrl2002}, in which the authors first introduced and studied linear elementary operators of algebras. Wang, in the article \cite{wang}, studied these three kinds of functions in a unified way by considering a new class of functions, and associated with each original kind of function a function in the new class, such that the original function is additive if and only if the associated function is zero.

Moreover, Jing in \cite{jingui} and \cite{jing} studied about the Jordan elementary maps, whose motivation also came from \cite{BresarandSemrl1999, Bresaretal2000, MolnarandSemrl2002}. In the following years, Ferreira with collaborators expanded these studies to several maps and non-associative structures, namely alternative rings (see \cite{6,9} and references therein) and Jordan algebras (see \cite{7,8} and references therein). Specifically, in \cite{8} the authors studied the additivity of $n$-multiplicative isomorphisms, as well as, $n$-multiplicative derivations between Jordan rings satisfying some conditions.

Axial algebras were introduced recently by Hall, Rehren and Shpectorov in \cite{13}. Important examples of axial algebras include Jordan algebras that are generated by idempotents \cite{16}. According to \cite{22}, the first studies about these types of algebras date back to the work of Albert \cite{1}. He showed that when a power-associative algebra $A$ is commutative, then $A = A_1(a)+A_0(a)+A_\frac{1}{2}(a)$, for any idempotent $a$. In particular, this is the case when $A$ is a Jordan algebra. Other classes of axial algebras are Griess algebras associated with vertex operator algebras (VOAs) and  Majorana algebras (see references in \cite{13}). The Griess algebra is a $196884$-dimensional non-associative algebra over $\mathbb{R}$. Norton showed that it contains idempotents, which he called axes that generate the algebra \cite{3}. VOAs were first introduced by physicists but particularly became of interest to mathematicians with Frenkel, Lepowsky, and Meurman's \cite{12} construction of the moonshine VOA $V^\natural$, whose automorphism group is the Monster Group $M$, also known as the Fischer-Griess monster, the largest sporadic finite simple group. The rigorous theory of VOAs was developed by Borcherds \cite{2} and it was instrumental in his proof of the monstrous moonshine conjecture. Majorana algebras are the predecessors of axial algebras. They were introduced by Sasha Ivanov \cite{14} to axiomatize some key properties of $V^\natural$. Majorana algebras were introduced to cover subalgebras of the Griess algebra (and some others), whereas
axial algebras are a wider class of algebras. We refer the reader to the introduction of \cite{13} for further information.

Throughout, let $\mathbb{F}$ be a field. An algebra is a vector space $A$ over $\mathbb{F}$ with a multiplication $\cdot : A \times A \rightarrow A$ that is bilinear. We will always consider algebras which are finite-dimensional here. We do not assume that our algebra has an identity, or that there are multiplicative inverses. In fact, the algebras we will consider will almost never have an identity. They will be commutative, but non-associative, by which we mean that they are not necessarily associative. That is, in general $x(yz) \neq (xy)z$ for $x, y, z$ in the algebra. Non-associative algebras can be unintuitive, difficult to work with, and very little can be said about them because of the non-associativity. However, there are classes of non-associative algebras which we can work with; these have some extra structure that allows us to get a handle on them. Axial algebras will be algebras generated by some idempotents which we call axes. Just as in the Jordan algebra case, we will ask that the algebra has a Peirce-like decomposition into a sum of eigenspaces. It is this and most importantly that the multiplication satisfies the fusion law which will give our algebras enough structure to be able to work with
them.

In studying preservers on algebras or rings, one usually assumes additivity in advance. Recently, however, a growing number of papers began investigating preservers that are not necessarily additive. Characterizing the interrelation between the multiplicative and additive structures of a ring or algebra in an interesting topic. The first result about the additivity of maps on rings was given by Martindale III \cite{18}. He established a condition on a ring $R$ such that every multiplicative isomorphism on $R$ is additive.

Our goal is to study the additivity of several kinds of functions, namely $n$-multiplicative isomorphisms, $n$-multiplicative derivations, elementary maps and Jordan elementary maps in a class of algebras that includes Jordan algebras with idempotents, $\mathcal{J}(\alpha)$-axial algebras and $\mathcal{M}(\alpha, \beta)$-axial algebras, and will be presented in Section \ref{axial}. For this, we will present in Section \ref{nullifying} a class of functions that are similar to those functions introduced in \cite{wang} and we will call them nullifying functions. In Sections \ref{j-algebra} and \ref{m-algebra} we will prove that such functions are zero under conditions similar to the Martindale conditions. Therefore, we will extend the results of \cite{FerreiraSmiglyBarreiro} and also answer its question.

\section{Axial Algebras}\label{axial}

In this section, we fix a field $\mathbb{F}$. The Peirce decomposition of Jordan algebras with idempotent lead us to a class of commutative algebras (not necessarily associative) with idempotents that we will call algebras over fusion laws. We also recall some concepts which are essential for the discussion of axial algebras (see reference \cite{13}).

Let $A$ be a commutative algebra over $\mathbb{F}$ and let $e\in A$. Then we define $L_e:A\rightarrow A$ by $L_e(x) = ex$ for all $x \in A$. For $\lambda \in \mathbb{F}$, the corresponding $\lambda$-eigenspace of $L_e$ is given by:
\begin{equation*}
A_\lambda = A_\lambda(e) = \{x \in A : L_e(x) = \lambda x\}.
\end{equation*}
If $S \subseteq \mathbb{F}$, then we will also write $A_S$ for the sum of all $A_\lambda$ with $\lambda \in S$.

\begin{definition}\label{d2.1}
A \textbf{fusion law} is a function $\mathcal{F}$ that assigns to each pair $(i,j)\in\mathbb{F}\times\mathbb{F}$ a subset $\mathcal{F}_{i,j}\subseteq\mathbb{F}$ and satisfies the following:
\begin{equation*}
\mathcal{F}_{i,j}=\mathcal{F}_{j,i}
\end{equation*}
for all $i,j\in\mathbb{F}$.
\end{definition}

We usually present a fusion law $\mathcal{F}$ by a finite table, where $\mathcal{F}_{i,j}=\emptyset$ if $(i,j)$ is outside the table. We define important examples of fusion laws: the associative-commutative-type law $\mathcal{A}$, the Jordan-type law $\mathcal{J}(\eta)$ and the (generalized) Monster-type law $\mathcal{M}(\alpha,\beta)$.

\begin{equation*}
\begin{array}{ccccc}
\begin{array}{|c||c|c|} 
\hline
\mathcal{A}& 1& 0\\
\hline
\hline
1& 1&\\
\hline
0&& 0\\
\hline
\end{array}
&&
\begin{array}{|c||c|c|c|} 
\hline
\mathcal{J}(\eta)& 1& 0& \eta\\
\hline
\hline
1& 1&& \eta\\ 
\hline
0&& 0& \eta\\ 
\hline
\eta& \eta& \eta& 1,0\\
\hline
\end{array}
&&
\begin{array}{|c||c|c|c|c|} 
\hline
\mathcal{M}(\alpha,\beta)& 1& 0& \alpha & \beta \\
\hline
\hline
1& 1&& \alpha & \beta \\
\hline
0&& 0& \alpha & \beta \\
\hline
\alpha & \alpha & \alpha & 1,0& \beta \\
\hline
\beta & \beta& \beta& \beta & 1,0,\alpha \\
\hline
\end{array}
\end{array}
\end{equation*}

\begin{definition}\label{d2.2}
Let $\mathcal{F}$ a fusion law. An \textbf{$\mathcal{F}$-algebra} is a pair $(A,e)$, where $A$ is a commutative algebra $A$ and $e\in A$ is an idempotent, such that:
\begin{itemize}
\item $A=\bigoplus_{i\in\mathbb{F}}A_i$,
\item $A_iA_j\subseteq A_{\mathcal{F}_{i,j}}$ for all $i,j\in\mathbb{F}$.
\end{itemize}
The $\mathcal{F}$-algebra is said to be \textbf{primitive} if $A_1 = \mathbb{F}e$.
\end{definition}

If $\mathrm{Char}(\mathbb{F})\neq2$, then a Jordan algebra with an idempotent is a $\mathcal{J}(\frac{1}{2})$-algebra. For other important examples of algebras over fusion laws, we will mention the axial algebras. Moreover, we may have $A_\lambda=0$, therefore an $\mathcal{A}$-algebra is a $\mathcal{J}(\eta)$-algebra for all $\eta\in\mathbb{F}\setminus\{1,0\}$ and likewise a $\mathcal{J}(\eta)$-axial algebra is an $\mathcal{M}(\eta,\beta)$-algebra for all $\beta \in \mathbb{F}\setminus\{1,0,\eta\}$.

\begin{definition}
Let $\mathcal{F}$ be a fusion law. An \textbf{$\mathcal{F}$-axial algebra} is a pair $(A,X)$, where $A$ is a commutative algebra and $X$ is a set of idempotent elements of $A$, such that:
\begin{itemize}
\item For every $e\in X$, then $(A,e)$ is an $\mathcal{F}$-algebra,
\item $A$ is generated as an algebra by $X$.
\end{itemize}
The elements of $X$ are called \textbf{axes}, and we say that $(A, X)$ is \textbf{primitive} if $(A,e)$ is primitive for all axes $e$.
\end{definition}

All the axes in the generating set $X$ satisfy the same fusion law $\mathcal{F}$. However, although two axes $e,e'\in X$ have the same fusion law, we do not assume that $A_\lambda (e)$ and $A_\lambda (e')$ have the same dimension. Now let us present some examples of axial algebras.

\begin{example}
The motivating example is the \textbf{Griess algebra}. It is a $196.884$-dimensional $\mathbb{R}$-algebra. Its automorphism group is the \textbf{Monster group}, an important group in the classification of finite simples groups. Griess introduced it in 1980 in the context of what is called \textbf{monstruous moonshine} and used it in 1982 to exhibit a construction of the Monster group. Norton showed that it is an $\mathcal{M}(\frac{1}{4},\frac{1}{32})$-axial algebra (see \cite{3}). For further details, see, for instance, \cite{3,12,13,14}.
\end{example}

\begin{example}
Another example is given by the \textbf{Matsuo algebras} $M_\eta(\Gamma)$. They are defined in terms of what are called \textbf{Fischer spaces} $\Gamma$, that are defined in terms of a group $(G,D)$ of $3$-transpositions (a group generated by involutions such that $\abs{ab}\leq 3$ for any generating involutions $a$ and$b$). Examples of groups of $3$-transpositions include symplectic groups, unitary groups and orthogonal groups in characteristic two and Fischer groups $\mathrm{Fi}_{22}$, $\mathrm{Fi}_{23}$ and $\mathrm{Fi}_{24}$. The set $D$ may be taken as a set of axes for $M_\eta(\Gamma)$ and they satisfy the fusion law $\mathcal{J}(\eta)$. Moreover the set of elements of the form $x=a+b$, where $a,b\in D$ are such that $ab=0$, may be taken as a set of axes for $M_\eta(\Gamma)$ and they satisfy the fusion law $\mathcal{M}(2\eta,\eta)$. For more details, see, for instance, \cite{galt}.
\end{example}

\begin{example}
We already said that Jordan algebras with idempotents are $\mathcal{J}(\frac{1}{2})$-algebras. But it is worth to point here that all simple Jordan algebras over an algebraically closed fields are generated by idempotents, so that they are examples of $\mathcal{J}(\frac{1}{2})$-axial algebras.
\end{example}

\begin{example}
An important family of $\mathcal{M}(\alpha,\beta)$-axial algebras are the Norton-Sakuma algebras, which are primitive $2$-generated $\mathcal{M}(\frac{1}{4},\frac{1}{32})$-axial algebras over a field of characteristic $0$ (see \cite{15}).
\end{example}

\section{Nullifying Functions}\label{nullifying}

We will define the types of functions that constitute the main theme of this paper.

\begin{definition}\label{d2.5}
Let $A$ and $A'$ be $\mathcal{F}$-algebras and let $\varphi:A\rightarrow A'$ be a bijective map. For a positive integer $n$, we will call $\varphi$ an \textbf{$n$-multiplicative isomorphism} from $A$ to $A'$ if for all $t_1,\dots,t_{n-1},x \in A$:
\begin{equation*}
\varphi(Lx) = \varphi(L)\varphi(x),
\end{equation*}
where $L=L_{t_1}\cdots L_{t_{n-1}}$ and $\varphi(L)=L_{\varphi(t_1)}\cdots L_{\varphi(t_{n-1})}$.
\end{definition}

\begin{definition}\label{d2.6}
Let $A$ be an $\mathcal{F}$-algebra and let $d:A\rightarrow A$ be a map. For a positive integer $n$, we will call $d$ an \textbf{$n$-multiplicative derivation} in $A$ if for all $t_1,\dots,t_{n-1},x \in A$:
\begin{equation*}
d(Lx) = d(L)x + Ld(x),
\end{equation*}
where $L=L_{t_1}\cdots L_{t_{n-1}}$ and $d(L)=(L_{d(t_1)}\cdots L_{t_{n-1}})+\cdots+(L_{t_1}\cdots L_{d(t_{n-1})})$.
\end{definition}

\begin{definition}
Let $A$ and $A'$ be $\mathcal{F}$-algebras and let $M:A\rightarrow A'$ and $M^*:A'\rightarrow A$ be maps. We will call the ordered pair $(M,M^*)$ an \textbf{elementary function} from $A$ to $A'$ if for any $a,b\in A$ and $x,y\in A'$:
\begin{equation*}
\begin{array}{rcl}
M(a(M^*(x)b)) &=& M(a)(xM(b)),\\
M^*(x(M(a)y)) &=& M^*(x)(aM^*(y)).
\end{array}
\end{equation*}
We say that $(M,M^*)$ is \textbf{surjective} (resp. \textbf{additive}) if $M$ and $M^*$ are surjective (resp. additive).
\end{definition}

\begin{definition}
Let $A$ and $A'$ be $\mathcal{F}$-algebras and let $M:A\rightarrow A'$ and $M^*:A'\rightarrow A$ be maps. We will call the ordered pair $(M,M^*)$ a \textbf{Jordan elementary function} from $A$ to $A'$ if for any $a\in A$ and $x\in A'$:
\begin{equation*}
\begin{array}{rcl}
M(aM^*(x)+M^*(x)a)&=&M(a)x+xM(a),\\
M^*(M(a)x+xM(a))&=&aM^*(x)+M^*(x)a.
\end{array}
\end{equation*}
We say that $(M,M^*)$ is \textbf{surjective} (resp. \textbf{additive}) if $M$ and $M^*$ are surjective (resp. additive).
\end{definition}

All the four above kinds of functions can be studied by considering what we will call nullifying functions. The idea of these functions is that they measure how far the original function in question is from additivity.

\begin{definition}
Let $A$ be an $\mathcal{F}$-algebra and let us denote by $A^*$ be the set of all nonempty finite sequences of elements of $A$ (the \textbf{Kleene star} of $A$). A \textbf{nullifying} function is a function $f:A^*\rightarrow A$ such that:
\begin{itemize}
\item[(I)] For $s_1,\dots,s_n\in A$ and for $\sigma$ a permutation of $\{1,\dots,n\}$, then $f(s_1,\dots,s_n)=f(s_{\sigma(1)},\dots,s_{\sigma(n)})$,
\item[(II)] For $s_1,\dots,s_n,t_1,\dots,t_k\in A$, then we have $f(s_1,\dots,s_n,t_1+\cdots+t_k)=f(s_1,\dots,s_n,t_1,\dots,t_k)$ if and only if $f(t_1,\dots,t_k)=0$,
\item[(III)] For $x\in A$, then $f(x)=0$,
\item[(IV)] For $s_1,\dots,s_n\in A$, then $f(s_1,\dots,s_n,0)=f(s_1,\dots,s_n)$,
\item[(V)] There is a fixed positive integer $r$ such that, for all $t_1,\dots,t_r\in A$, the function $L=L_{t_1}\cdots L_{t_r}$ satisfies $L(f(s_1,\dots,s_n))=f(L(s_1),\dots,L(s_n))$ for $s_1,\dots,s_n\in A$.
\end{itemize}
\end{definition}

The following property follows easily by induction.

\begin{proposition}\label{add}
Let $f$ be an nullifying function. If $f(x,y)=0$ for any $x,y\in A$, then $f=0$.
\end{proposition}

The motivation for the properties (I) to (V) will be clear after presenting the following examples.

\begin{example}\label{isomorphism}
Let $A$ and $A'$ be $\mathcal{F}$-algebras and let $\varphi:A\rightarrow A'$ be an $n$-multiplicative isomorphism. Define:
\begin{equation*}
f(x_1,\dots,x_n)=\varphi^{-1}(\varphi(x_1+\cdots+x_n)-\varphi(x_1)-\cdots-\varphi(x_n)).
\end{equation*}
Then $f$ is an nullifying function.
\end{example}

\begin{example}\label{derivation}
Let $A$ be an $\mathcal{F}$-algebra and $d:A\rightarrow A$ be an $n$-multiplicative derivation. Define:
\begin{equation*}
f(x_1,\dots,x_n)=d(x_1+\cdots+x_n)-d(x_1)-\cdots-d(x_n).
\end{equation*}
Then $f$ is an nullifying function.
\end{example}

\begin{example}\label{elementary}
Let $A$ and $A'$ be $\mathcal{F}$-algebras and let $(M, M^*)$ be an elementary function from $A$ to $A'$ such that $M$ is bijective and $M^*$ is surjective. Define:
\begin{equation*}
f(x_1,\dots,x_n)=M^{-1}(M(x_1+\cdots+x_n)-M(x_1)-\cdots-M(x_n)).
\end{equation*}
Then $f$ is an nullifying function.
\end{example}

\begin{example}\label{jordan-elementary}
Let $A$ and $A'$ be $\mathcal{F}$-algebras such that $\mathrm{Char}(\mathbb{F})\neq2$ and let $(M, M^*)$ be a Jordan elementary function from $A$ to $A'$ such that $M$ is bijective and $M^*$ is surjective. Define:
\begin{equation*}
f(x_1,\dots,x_n)=M^{-1}(M(x_1+\cdots+x_n)-M(x_1)-\cdots-M(x_n)).
\end{equation*}
Then $f$ is an nullifying function.
\end{example}

Therefore, in each of the previous examples, the original function is additive if and only if the associated nullifying function is zero everywhere.

\section{$\mathcal{J}(\alpha)$-algebras}\label{j-algebra}

In this section we will consider the $\mathcal{J}(\alpha)$-algebras, that include Jordan algebras with idempotents and $\mathcal{J}(\alpha)$-axial algebras.

Let $A$ be an $\mathcal{J}(\alpha)$-algebra. Then $A$ has a direct sum decomposition $A = A_1 \oplus A_0 \oplus A_\alpha$. The following three conditions are very important to prove one of our main results. They are the same as conditions (i) to (iii) about Jordan algebras with idempotents in \cite{7,8,ji-jordan} and conditions (i) to (iii) about $\mathcal{J}(\alpha)$-axial algebras in \cite{FerreiraSmiglyBarreiro}, where they were called Martindale-like conditions.
\begin{itemize}
\item[(i)] For $a_i \in A_i$ ($i = 1, 0$), if $t_\alpha a_i = 0$ for all $t_\alpha \in A_\alpha$, then $a_i = 0$,
\item[(ii)] For $a_0 \in A_0$, if $t_0a_0 = 0$ for all $t_0 \in A_0$, then $a_0 = 0$,
\item[(iii)] For $a_\alpha \in A_\alpha$, if $t_0a_\alpha = 0$ for all $t_0 \in A_0$, then $a_\alpha = 0$.
\end{itemize}
Assuming that the $\mathcal{J}(\alpha)$-algebra $A$ satisfies the above three sufficient conditions, we prove that every nullifying function is zero. Hence, we generalize to the context of $\mathcal{J}(\alpha)$-algebras the results proved for Jordan algebras with idempotents and $\mathcal{J}(\alpha)$-axial algebras. We will present examples of $\mathcal{J}(\alpha)$-algebras that satisfy the Martindale-like conditions. The first one is a Jordan algebra.

\begin{example}
Let $\mathbb{F}$ be a field such that $\mathrm{Char}(\mathbb{F})\neq2$ and let $B$ be an algebra of dimension $4$ over $\mathbb{F}$ with basis $\{e_{11},e_{10},e_{01},e_{00}\}$ and multiplication table given by:
\begin{equation*}
e_{ij}e_{kl} = \delta_{jk}e_{il},
\end{equation*}
for $i,j,k,l\in\{1,0\}$, where $\delta_{jk}$ is the Kronecker delta function. Consider the Jordan algebra $A=B^{(+)}$ with the same vector space as $B$ and multiplication defined by $x*y=\frac{1}{2}(xy+yx)$. Then $A$ is a Jordan algebra and the element $e_{11}$ is idempotent. So $(A,e_{11})$ is a $\mathcal{J}(\frac{1}{2})$-algebra and we have the following decomposition:
\begin{equation*}
A_1=\mathbb{F}e_{11},\quad\quad A_{\frac{1}{2}}=\mathbb{F}e_{10}+\mathbb{F}e_{01},\quad\quad A_0=\mathbb{F}e_{00}.
\end{equation*}
It can be shown that this $\mathcal{J}(\frac{1}{2})$-algebra satisfies the Martindale-like conditions (i) to (iii).
\end{example}

The following one is a $\mathcal{J}(\alpha)$-axial algebra.

\begin{example}
Let $\mathbb{F}$ be a field such that $\mathrm{Char}(\mathbb{F})\neq2,3,5$ and let $A$ be an algebra of dimension $3$ over $\mathbb{F}$ with basis $\{e_A,e_B,e_C\}$ and multiplication table given by:
\begin{equation*}
\begin{array}{|c||c|c|c|}
\hline
&e_A&e_B&e_C
\\\hline\hline
e_A&e_A&\frac{1}{8}(e_A + e_B - e_C)&\frac{1}{8}(e_A - e_B + e_C)
\\\hline
e_B&\frac{1}{8}(e_A + e_B - e_C)&e_B&\frac{1}{8}(- e_A + e_B + e_C)
\\\hline
e_C&\frac{1}{8}(e_A - e_B + e_C)&\frac{1}{8}(- e_A + e_B + e_C)&e_C\\
\hline
\end{array}
\end{equation*}
Then $(A,e_A)$ is a $\mathcal{J}(\frac{1}{4})$-algebra and we have the following decomposition:
\begin{equation*}
\begin{array}{c}
A_1 = \mathbb{F}e_A,\quad\quad A_0 = \mathbb{F}(-e_A+4e_B+4e_C),\quad\quad A_\frac{1}{4} = \mathbb{F}(e_B - e_C).
\end{array}
\end{equation*}
We also can show that $(A,\{e_A,e_B,e_C\})$ is a $\mathcal{J}(\frac{1}{4})$-axial algebra. This algebra is called a Norton-Sakuma algebra of the type 2A, by the article \cite{15}. Moreover, this $\mathcal{J}(\frac{1}{4})$-álgebra satisfies the Martindale-like conditions (i) to (iii).
\end{example}

Now we present the main result.

\begin{theorem}\label{3.1}
Let $(A,e)$ be a $\mathcal{J}(\alpha)$-algebra that satisfies the following conditions:
\begin{itemize}
\item[(i)] For $a_i \in A_i$ ($i = 1, 0$), if $t_\alpha a_i = 0$ for all $t_\alpha \in A_\alpha$, then $a_i = 0$,
\item[(ii)] For $a_0 \in A_0$, if $t_0a_0 = 0$ for all $t_0 \in A_0$, then $a_0 = 0$,
\item[(iii)] For $a_\alpha \in A_\alpha$, if $t_0a_\alpha = 0$ for all $t_0 \in A_0$, then $a_\alpha = 0$.
\end{itemize}
Then every nullifying function is zero.
\end{theorem}

Throughout the remainder of this subsection, we will assume the hypotheses of Theorem \ref{3.1}. Let $f$ be a nullifying function and let $r$ be a positive integer such that, for all $t_1,\dots,t_r\in A$, the function $L=L_{t_1}\cdots L_{t_r}$ satisfies:
\begin{equation*}
L(f(s_1,\dots,s_n))=f(L(s_1),\dots,L(s_n)),
\end{equation*}
for all $s_1,\dots,s_n\in A$. Moreover we will define the following sets of functions:
\begin{itemize}
\item Let $\mathcal{L}_0$ consist of all functions of the form $L_{t_1}\cdots L_{t_r}$, where $t_1,\dots,t_r\in A_0$,
\item Let $\mathcal{L}_1$ consist of all functions of the form $L_e^r$,
\item Let $\mathcal{L}_\alpha$ consist of all functions of the form $L_e^{r-1}L_t$, where $t\in A_\alpha$.
\end{itemize}
Therefore, by the fusion law $\mathcal{J}(\alpha)$:
\begin{equation*}
\mathcal{L}_iA_j\subseteq A_{\mathcal{J}(\alpha)_{i,j}}.
\end{equation*}
Because $\alpha\neq0$, we have the following properties:
\begin{itemize}
\item[(i)] For $a_i \in A_i$ ($i = 1,0$), if $\mathcal{L}_\alpha a_i = 0$, then $a_i = 0$,
\item[(ii)] For $a_0 \in A_0$, if $\mathcal{L}_0a_0 = 0$, then $a_0 = 0$,
\item[(iii)] For $a_\alpha \in A_\alpha$, if $\mathcal{L}_0a_\alpha = 0$, then $a_\alpha = 0$.
\end{itemize}
Now we define the following:
\begin{itemize}
\item $\mathcal{F}_\alpha=\mathcal{L}_0\mathcal{L}_1$.
\item $\mathcal{F}_1=\mathcal{F}_\alpha\mathcal{L}_\alpha\mathcal{L}_1$.
\item $\mathcal{F}_0=\mathcal{F}_\alpha\mathcal{L}_\alpha\mathcal{L}_0$.
\end{itemize}
Now we show several lemmas.

\begin{lemma}\label{jaxial}
Let $i,j\in\{0,1,\alpha\}$. If $i\neq j$, then $\mathcal{F}_iA_j=0$. Also, for $a_i\in A_i$, if $\mathcal{F}_ia_i=0$, then $a_i=0$.
\end{lemma}
\begin{proof}
We divide the proof in steps.

\medskip
\noindent
\textbf{Step 1:} Consider the set $\mathcal{F}_\alpha$. By the fusion law $\mathcal{J}(\alpha)$ we have:
\begin{equation*}
\mathcal{F}_\alpha A_0
=\mathcal{L}_0\mathcal{L}_1A_0
=\mathcal{L}_00
=0,
\end{equation*}
and also:
\begin{equation*}
\mathcal{F}_\alpha A_1
=\mathcal{L}_0\mathcal{L}_1A_1
\subseteq\mathcal{L}_0A_1
=0.
\end{equation*}
Moreover, we have the following:
\begin{equation*}
\begin{array}{rcl}
a_\alpha\in A_\alpha\setminus\{0\}&\overset{\alpha\neq0}{\Rightarrow}&\exists L_1\in\mathcal{L}_1:L_1a_\alpha\in A_\alpha\setminus\{0\}\\
&\overset{\text{(iii)}}{\Rightarrow}&\exists L_0\in\mathcal{L}_0:L_0L_1a_\alpha\in A_\alpha\setminus\{0\}.
\end{array}
\end{equation*}

\medskip
\noindent
\textbf{Step 2:} Consider the set $\mathcal{F}_1$. Then:
\begin{equation*}
\mathcal{F}_1A_0
=\mathcal{F}_\alpha\mathcal{L}_\alpha\mathcal{L}_1A_0
=\mathcal{F}_\alpha\mathcal{L}_\alpha0
=0,
\end{equation*}
and also:
\begin{equation*}
\mathcal{F}_1A_\alpha
=\mathcal{F}_\alpha\mathcal{L}_\alpha\mathcal{L}_1A_\alpha
\subseteq\mathcal{F}_\alpha\mathcal{L}_\alpha A_\alpha
\subseteq\mathcal{F}_\alpha(A_1+A_0)
\overset{\text{S1}}{=}0.
\end{equation*}
Now we have the following:
\begin{equation*}
\begin{array}{rcl}
a_1\in A_1\setminus\{0\}&\Rightarrow&\exists L_1\in\mathcal{L}_1:L_1a_1\in A_1\setminus\{0\}\\
&\overset{\text{(i)}}{\Rightarrow}&\exists L_\alpha\in\mathcal{L}_\alpha:L_\alpha L_1a_1\in A_\alpha\setminus\{0\}\\
&\overset{\text{S1}}{\Rightarrow}&\exists F_\alpha\in\mathcal{F}_\alpha:F_\alpha L_\alpha L_1a_1\in A_\alpha\setminus\{0\}.
\end{array}
\end{equation*}

\medskip
\noindent
\textbf{Step 3:} Consider the set $\mathcal{F}_0$. Then:
\begin{equation*}
\mathcal{F}_0A_1
=\mathcal{F}_\alpha\mathcal{L}_\alpha\mathcal{L}_0A_1
=\mathcal{F}_\alpha\mathcal{L}_\alpha0
=0,
\end{equation*}
and also:
\begin{equation*}
\mathcal{F}_0A_\alpha
=\mathcal{F}_\alpha\mathcal{L}_\alpha\mathcal{L}_0A_\alpha
\subseteq\mathcal{F}_\alpha\mathcal{L}_\alpha A_\alpha
\subseteq\mathcal{F}_\alpha(A_1+A_0)\overset{\text{S1}}{=}0.
\end{equation*}
Now we have the following:
\begin{equation*}
\begin{array}{rcl}
a_0\in A_0\setminus\{0\}&\overset{\text{(ii)}}{\Rightarrow}&\exists L_0\in\mathcal{L}_0:L_0a_0\in A_0\setminus\{0\}\\
&\overset{\text{(i)}}{\Rightarrow}&\exists L_\alpha\in\mathcal{L}_\alpha:L_\alpha L_0a_0\in A_\alpha\setminus\{0\}\\
&\overset{\text{S1}}{\Rightarrow}&\exists F_\alpha\in\mathcal{F}_\alpha:F_\alpha L_\alpha L_0a_0\in A_\alpha\setminus\{0\}.
\end{array}
\end{equation*}
Therefore, we conclude the proof.
\end{proof}

\begin{lemma}\label{3.2}
Let $a_i \in A_i$, where $i \in \{1, 0, \alpha\}$. Then
$f(a_1,a_0,a_\alpha)=0$.
\end{lemma}
\begin{proof}
Define:
\begin{equation*}
s = f(a_1,a_0,a_\alpha).
\end{equation*}
Write $s = s_1 + s_0 + s_\alpha \in A$, where $s_i \in A_i$ for $i \in \{1, 0, \alpha\}$. Now let $i\in\{1,0,\alpha\}$. We will show that $s_i=0$. Indeed:
\begin{equation*}
\mathcal{F}_is_i\overset{\text{L\ref{jaxial}}}{=}\mathcal{F}_is\overset{\text{(V)}}{=}f(\mathcal{F}_ia_1,\mathcal{F}_ia_0,\mathcal{F}_ia_\alpha)\overset{\text{L\ref{jaxial}}}{=}f(\mathcal{F}_ia_i,0,0)=0.
\end{equation*}
Thus, by Lemma \ref{jaxial} we have $s_i = 0$.
\end{proof}

\begin{lemma}\label{3.3}
Let $a_0 \in A_0$ and let $a_\alpha, b_\alpha \in A_\alpha$. Then $f(a_\alpha a_0, b_\alpha) = 0$.
\end{lemma}
\begin{proof}
Consider the function $L=L_{\alpha^{-1}e}^{r-1}L_{\alpha^{-1}e+a_\alpha}$. Then $La_0=a_\alpha a_0$ and $Lb_\alpha=b_\alpha+L_{\alpha^{-1}e}^{r-1}(a_\alpha b_\alpha)$. But $L_{\alpha^{-1}e}^{r-1}(a_\alpha b_\alpha)\in A_1+A_0$, so there are $b_1\in A_1$ and $b_0\in A_0$ such that $L_{\alpha^{-1}e}^{r-1}(a_\alpha b_\alpha)=b_1+b_0$. Thus, by successive applications of Lemma \ref{3.2}, we have:
\begin{equation*}
\begin{array}{rcl}
0
&\overset{\text{L\ref{3.2}}}{=}& Lf(a_0,b_\alpha)\\
&\overset{\text{(V)}}{=}& f(La_0,Lb_\alpha)\\
&=& f(a_\alpha a_0,b_\alpha+b_1+b_0)\\
&\overset{\text{L\ref{3.2}}}{=}& f(a_\alpha a_0,b_\alpha,b_1+b_0)
\end{array}
\end{equation*}
Also, by Lemma \ref{3.2} we have:
\begin{equation*}
f(a_\alpha a_0+b_\alpha,b_1+b_0)=0.
\end{equation*}
Therefore, by property (II) we have $f(a_\alpha a_0, b_\alpha) = 0$.
\end{proof}

\begin{lemma}\label{3.4}
Let $a_i, b_i \in A_i$, where $i\in\{1,0,\alpha\}$. Then $f(a_i, b_i) = 0$.
\end{lemma}
\begin{proof}
We divide the proof in steps.

\medskip
\noindent
\textbf{Step 1:} Let $s = f(a_\alpha,b_\alpha)$ and write $s = s_1 + s_0 + s_\alpha \in A$, where $s_i \in A_i$ for $i \in \{1, 0, \alpha\}$. For $i\in\{0,1\}$, we have:
\begin{equation*}
\mathcal{F}_is_i\overset{\text{L\ref{jaxial}}}{=}\mathcal{F}_is\overset{\text{(V)}}{=}f(\mathcal{F}_ia_\alpha,\mathcal{F}_ib_\alpha)\overset{\text{L\ref{jaxial}}}{=}f(0,0)=0,
\end{equation*}
so, by Lemma \ref{jaxial} we have $s_i=0$. Now, for $F_\alpha\in\mathcal{F}_\alpha$, then $F_\alpha a_\alpha=a'_\alpha a'_0$, where $a'_\alpha\in A_\alpha$ and $a'_0\in A_0$, and also $F_\alpha b_\alpha=b'_\alpha\in A_\alpha$, so:
\begin{equation*}
F_\alpha s_\alpha\overset{\text{L\ref{jaxial}}}{=}F_\alpha s\overset{\text{(V)}}{=}f(F_\alpha a_\alpha,F_\alpha b_\alpha)=f(a'_\alpha a'_0,b'_\alpha)\overset{\text{L\ref{3.3}}}{=}0,
\end{equation*}
hence, by Lemma \ref{jaxial}, we have $s_\alpha=0$.

\medskip
\noindent
\textbf{Step 2:} Let $i\in\{1,0\}$, let $s = f(a_i,b_i)$ and write $s = s_1 + s_0 + s_\alpha \in A$, where $s_i \in A_i$ for $i \in \{1, 0, \alpha\}$. For $j\in\{0,1,\alpha\}\setminus\{i\}$, we have:
\begin{equation*}
\mathcal{F}_js_j\overset{\text{L\ref{jaxial}}}{=}\mathcal{F}_js\overset{\text{(V)}}{=}f(\mathcal{F}_ja_i,\mathcal{F}_jb_i)\overset{\text{L\ref{jaxial}}}{=}f(0,0)=0,
\end{equation*}
thus, by Lemma \ref{jaxial} we have $s_j=0$. Now:
\begin{equation*}
\mathcal{F}_i s_i\overset{\text{L\ref{jaxial}}}{=}\mathcal{F}_i s\overset{\text{(V)}}{=}f(\mathcal{F}_i a_i,\mathcal{F}_i b_i)\subseteq f(A_\alpha,A_\alpha)\overset{\text{S1}}{=}0,
\end{equation*}
hence, by Lemma \ref{jaxial}, we have $s_i=0$.
\end{proof}

\begin{lemma}
$f=0$.
\end{lemma}
\begin{proof}
For $x,y\in A$ we have:
\begin{equation*}
\begin{array}{rcl}
f(x,y)
&=& f(x_1+x_0+x_\alpha,y_1+y_0+y_\alpha)\\
&\overset{\text{L\ref{3.2}}}{=}& f(x_1,x_0,x_\alpha,y_1,y_0,y_\alpha)\\
&=& f(x_1,y_1,x_0,y_0,x_\alpha,y_\alpha)\\
&\overset{\text{L\ref{3.4}}}{=}& f(x_1+y_1,x_0+y_0,x_\alpha+y_\alpha)\\
&\overset{\text{L\ref{3.2}}}{=}& 0.
\end{array}
\end{equation*}
Now apply the Proposition \ref{add}.
\end{proof}

\begin{corollary}\label{c3}
Let $(A,e)$ a $\mathcal{J}(\alpha)$-algebra that satisfies the following conditions:
\begin{itemize}
\item[(i)] For $a_i \in A_i$ ($i = 1, 0$), if $t_\alpha a_i = 0$ for all $t_\alpha \in A_\alpha$, then $a_i = 0$,
\item[(ii)] For $a_0 \in A_0$, if $t_0a_0 = 0$ for all $t_0 \in A_0$, then $a_0 = 0$,
\item[(iii)] For $a_\alpha \in A_\alpha$, if $t_0a_\alpha = 0$ for all $t_0 \in A_0$, then $a_\alpha = 0$.
\end{itemize}
Then:
\begin{itemize}
\item[(a)] Every $n$-multiplicative isomorphism is additive,
\item[(b)] Every $n$-multiplicative derivation is additive,
\item[(c)] Every surjective elementary function is additive,
\item[(d)] If $\mathrm{Char}(\mathbb{F})\neq2$, then every surjective Jordan elementary function is additive.
\end{itemize}
\end{corollary}
\begin{proof}
The items \textbf{(a)} and \textbf{(b)} follow from Examples \ref{isomorphism} and \ref{derivation}.

\medskip
\noindent
\textbf{c)} By Example \ref{elementary}, it suffices to show that $M$ and $M^*$ are injective and note that the pair $((M^*)^{-1},M^{-1})$ is an elementary function from $A$ to $B$.

We will show that $M$ is injective. For any $x,u,v\in A$, there are $u',v'\in B$ such that $M^*(u')=u$ and $M^*(v')=v$, so:
\begin{equation*}
u(vx)=u(xv)=M^*(u')(xM^*(v'))=M^*(u'(M(x)v')).
\end{equation*}
Now let $x,y\in A$ such that $M(x)=M(y)$. Then for any $u,v\in A$ we have $u(vx)=u(vy)$. Let $s=x-y$ and recall the definitions of $\mathcal{F}_\alpha$, $\mathcal{F}_1$ and $\mathcal{F}_0$ where $r=1$. For $i\in\{1,0,\alpha\}$, then by Lemma \ref{jaxial} we have $\mathcal{F}_is_i=\mathcal{F}_is=0$, so by Lemma \ref{jaxial} we have $s_i=0$. Therefore $s=0$, so $x=y$.

Now we will show that $M^*$ is injective. For any $x,y\in A$ and $u\in B$, there are $x',y'\in A$ such that $M^*M(x')=x$ and $M^*M(y')=y$, so that:
\begin{equation*}
\begin{array}{rcl}
x(yM^{-1}(u))&=&x(M^{-1}(u)y)\\
&=&M^*M(x')(M^{-1}(u)M^*M(y))\\
&=&M^*(M(x')(uM(y')))\\
&=&M^*M(x(M^*(u)y)).
\end{array}
\end{equation*}
Now let $u,v\in B$ such that $M^*(u)=M^*(v)$. Then for $x,y\in A$ we have $x(yM^{-1}(u))=x(yM^{-1}(v))$. Analogously to the preceding paragraph, we can show that $M^{-1}(u)=M^{-1}(v)$, so $u=v$.

\medskip
\noindent
\textbf{d)} By Example \ref{jordan-elementary}, it suffices to show that $M$ and $M^*$ are injective and note that the pair $((M^*)^{-1},M^{-1})$ is a Jordan elementary function from $A$ to $B$.

We will show that $M$ is injective. For ant $x,u\in A$, there are $u'\in B$ such that $M^*(u')=\frac{1}{2}u$, so that:
\begin{equation*}
ux=M^*(u')x+xM^*(u')=M^*(u'M(x)+M(x)u').
\end{equation*}
Now let $x,y\in A$ such that $M(x)=M(y)$. Then for any $u\in A$ we have $ux=uy$. Thus, analogously to item \textbf{(c)}, we obtain $x=y$.

Now we will show that $M^*$ is injective. For any $x\in A$ and $u\in B$, there is $x'\in A$ such that $M^*M(x')=\frac{1}{2}x$, so that:
\begin{equation*}
\begin{array}{rcl}
xM^{-1}(u)&=&M^*M(x')M^{-1}(u)+M^{-1}(u)M^*M(x')\\
&=&M^*(M(x')u+uM(x'))\\
&=&M^*M(x'M^*(u)+M^*(u)x').
\end{array}
\end{equation*}
Now let $u,v\in B$ such that $M^*(u)=M^*(v)$. Then for $x\in A$ we have $xM^{-1}(u)=xM^{-1}(v)$. Thus, analogously to item \textbf{(c)}, we can show that $M^{-1}(u)=M^{-1}(v)$, so $u=v$.
\end{proof}

\section{$\mathcal{M}(\alpha,\beta)$-algebras}\label{m-algebra}

In this section, we will consider the $\mathcal{M}(\alpha,\beta)$-algebras, that include $\mathcal{M}(\alpha,\beta)$-axial algebras.

Let $A$ be an $\mathcal{M}(\alpha,\beta)$-algebra. Then $A$ has a direct sum decomposition $A = A_1 \oplus A_0 \oplus A_\alpha \oplus A_\beta$. The following five conditions are crucial to prove our main results. They are the same as conditions (i) to (v) about $\mathcal{M}(\alpha,\beta)$-axial algebras in \cite{FerreiraSmiglyBarreiro}, where they were also called Martindale-like conditions.
\begin{itemize}
\item[(i)] For $a_i \in A_i$ ($i = 1, 0$), if $t_\alpha a_i = 0$ for all $t_\alpha \in A_\alpha$, then $a_i = 0$,
\item[(ii)] For $a_0 \in A_0$, if $t_0a_0 = 0$ for all $t_0 \in A_0$, then $a_0 = 0$,
\item[(iii)] For $a_i \in A_i$ ($i = \alpha, \beta$), if $t_0a_i = 0$ for all $t_0 \in A_0$, then $a_i = 0$,
\item[(iv)] For $a_\beta \in A_\beta$, if $t_\alpha a_\beta = 0$ for all $t_\alpha \in A_\alpha$, then $a_\beta = 0$,
\item[(v)] For $a_\alpha \in A_\alpha$, if $t_\beta a_\alpha = 0$ for all $t_\beta \in A_\beta$, then $a_\alpha = 0$.
\end{itemize}
Let us assume that the $\mathcal{M}(\alpha,\beta)$-algebra $A$ satisfies the above five sufficient conditions, we prove that every nullifying function is zero. Hence, we generalize to the context of $\mathcal{M}(\alpha,\beta)$-algebras the results proved for $\mathcal{M}(\alpha,\beta)$-axial algebras. We will show an example of an $\mathcal{M}(\alpha,\beta)$-axial algebra that satisfy the Martindale-like conditions.

\begin{example}
Let $\mathbb{F}$ be a field such that $\mathrm{Char}(\mathbb{F})\neq2,3$ and we define an algebra $H$, called \textbf{Highwater algebra} because it was discovered in Venice during the disastrous floods in November 2019. The algebra $H$ has a base $\{a_i, \sigma_j \mid i \in \mathbb{Z}, j \in \mathbb{Z}_{>0}\}$ and multiplication table given by:
\begin{itemize}
\item[$\bullet$] $a_ia_j = \frac{1}{2}(a_i + a_j) + \sigma_{\abs{i-j}}$,
\item[$\bullet$] $a_i\sigma_j = \sigma_ja_i = -\frac{3}{4}a_i + \frac{3}{8}(a_{i-j} + a_{i+j}) + \frac{3}{2}\sigma_j$,
\item[$\bullet$] $\sigma_i\sigma_j = \frac{3}{4}(\sigma_i + \sigma_j) - \frac{3}{8}(\sigma_{\abs{i-j}} + \sigma_{i+j})$,
\end{itemize}
where we define $\sigma_0=0$. In particular, $a^2_i = \frac{1}{2} (a_i + a_i) + \sigma_0 = a_i$, thus each $a_i$ is an idempotent. By the article \cite{highwater}, for all $i\in\mathbb{Z}$ then $(H,a_i)$ is an $\mathcal{M}$-algebra, where $\mathcal{M}$ is the following fusion law:
\begin{equation*}
\begin{array}{|c||c|c|c|c|}
\hline
\mathcal{M}&1&0&2&\frac{1}{2}\\
\hline
\hline
1&\{1\}&\emptyset&\{2\}&\{\frac{1}{2}\}\\
\hline
0&\emptyset&\{0\}&\{2\}&\{\frac{1}{2}\} \\
\hline
2&\{2\}&\{2\}&\{0\}&\{\frac{1}{2}\}\\
\hline
\frac{1}{2}&\{\frac{1}{2}\}&\{\frac{1}{2}\}&\{\frac{1}{2}\}&\{0,2\}\\
\hline
\end{array}
\end{equation*}
In particular, $(H,a_i)$ is an $\mathcal{M}(2,\frac{1}{2})$-algebra. Namely:
\begin{equation*}
\begin{array}{rcl}
H_1&=&\mathbb{F}a_i,\\
H_0&=&\sum\limits_{j=1}^\infty \mathbb{F}(6a_i - 3(a_{-j+i} + a_{j+i}) + 4\sigma_{j}),\\
H_2&=&\sum\limits_{j=1}^\infty \mathbb{F}(2a_i - (a_{-j+i} + a_{j+i}) - 4\sigma_{j}),\\
H_\frac{1}{2}&=&\sum\limits_{j=1}^\infty \mathbb{F}(a_{-j+i} - a_{j+i}).
\end{array}
\end{equation*}
Also, by Theorem 2.2 of \cite{highwater}, $(H,\{a_0,a_1\})$ is an $\mathcal{M}$-axial algebra. Now it can be shown that this $\mathcal{M}(2,\frac{1}{2})$-algebra $(H,a_0)$ satisfies the Martindale-like conditions (i) to (v).
\end{example}

Now we present the main result.

\begin{theorem}\label{4.1}
Let $A$ be an $\mathcal{M}(\alpha,\beta)$-algebra that satisfies the following conditions:
\begin{itemize}
\item[(i)] For $a_i \in A_i$ ($i = 1, 0$), if $t_\alpha a_i = 0$ for all $t_\alpha \in A_\alpha$, then $a_i = 0$,
\item[(ii)] For $a_0 \in A_0$, if $t_0a_0 = 0$ for all $t_0 \in A_0$, then $a_0 = 0$,
\item[(iii)] For $a_i \in A_i$ ($i = \alpha, \beta$), if $t_0a_i = 0$ for all $t_0 \in A_0$, then $a_i = 0$,
\item[(iv)] For $a_\beta \in A_\beta$, if $t_\alpha a_\beta = 0$ for all $t_\alpha \in A_\alpha$, then $a_\beta = 0$,
\item[(v)] For $a_\alpha \in A_\alpha$, if $t_\beta a_\alpha = 0$ for all $t_\beta \in A_\beta$, then $a_\alpha = 0$.
\end{itemize}
Then every nullifying function is zero.
\end{theorem}

Throughout the remainder of this subsection, we will assume the hypotheses of Theorem \ref{4.1}. Let $f$ be a nullifying function and let $r$ be a positive integer such that, for all $t_1,\dots,t_r\in A$, the function  $L=L_{t_1}\cdots L_{t_r}$ satisfies:
\begin{equation*}
L(f(s_1,\dots,s_n))=f(L(s_1),\dots,L(s_n)),
\end{equation*}
for all $s_1,\dots,s_n\in A$. Moreover we will define the following sets of functions:
\begin{itemize}
\item Let $\mathcal{L}_0$ consist of all functions of the form $L_{t_1}\cdots L_{t_r}$, where $t_1,\dots,t_r\in A_0$,
\item Let $\mathcal{L}_1$ consist of all functions of the form $L_e^r$,
\item Let $\mathcal{L}_\alpha$ consist of all functions of the form $L_e^{r-1}L_t$, where $t\in A_\alpha$,
\item Let $\mathcal{L}_\beta$ consist of all functions of the form $L_e^{r-1}L_t$, where $t\in A_\beta$.
\end{itemize}
Therefore, by the fusion law $\mathcal{M}(\alpha,\beta)$:
\begin{equation*}
\mathcal{L}_iA_j\subseteq\sum_{k\in\mathcal{M}(\alpha,\beta)_{i,j}}A_k.
\end{equation*}
Because $\alpha,\beta\neq0$, we have the following properties:
\begin{itemize}
\item[(i)] For $a_i \in A_i$ ($i = 1,0$), if $\mathcal{L}_\alpha a_i = 0$, then $a_i = 0$,
\item[(ii)] For $a_0 \in A_0$, if $\mathcal{L}_0a_0 = 0$, then $a_0 = 0$,
\item[(iii)] For $a_i \in A_i$ ($i = \alpha,\beta$), if $\mathcal{L}_0a_i = 0$, then $a_i = 0$,
\item[(iv)] For $a_\beta \in A_\beta$, if $\mathcal{L}_\alpha a_\beta = 0$, then $a_\beta = 0$,
\item[(v)] For $a_\alpha \in A_\alpha$, if $\mathcal{L}_\beta a_\alpha = 0$, then $a_\alpha = 0$.
\end{itemize}
Now we define the following:
\begin{itemize}
\item $\mathcal{G}=\mathcal{L}_0\mathcal{L}_1$.
\item $\mathcal{F}_\beta=\mathcal{G}\mathcal{L}_\alpha\mathcal{G}$.
\item $\mathcal{F}_\alpha=\mathcal{F}_\beta\mathcal{L}_\beta\mathcal{G}$.
\item $\mathcal{F}_1=\mathcal{F}_\alpha\mathcal{L}_\alpha\mathcal{L}_1$.
\item $\mathcal{F}_0=\mathcal{F}_\alpha\mathcal{L}_\alpha\mathcal{L}_0$.
\end{itemize}
Now we show several lemmas.

\begin{lemma}\label{maxial}
Let $i,j\in\{0,1,\alpha,\beta\}$. If $i\neq j$, then $\mathcal{F}_iA_j=0$. Also, for $a_i\in A_i$, if $\mathcal{F}_ia_i=0$, then $a_i=0$.
\end{lemma}
\begin{proof}
We divide the proof in steps.

\medskip
\noindent
\textbf{Step 1:} Consider the set $\mathcal{G}$. By the fusion law $\mathcal{M}(\alpha,\beta)$ we have:
\begin{equation*}
\mathcal{G} A_0
=\mathcal{L}_0\mathcal{L}_1A_0
=\mathcal{L}_00
=0,
\end{equation*}
and also:
\begin{equation*}
\mathcal{G} A_1
=\mathcal{L}_0\mathcal{L}_1A_1
\subseteq\mathcal{L}_0A_1
=0.
\end{equation*}
Now, for $i\in\{\alpha,\beta\}$, we have the following:
\begin{equation*}
\begin{array}{rcl}
a_i\in A_i\setminus\{0\}&\overset{i\neq0}{\Rightarrow}&\exists L_1\in\mathcal{L}_1:L_1a_i\in A_i\setminus\{0\}\\
&\overset{\text{(iii)}}{\Rightarrow}&\exists L_0\in\mathcal{L}_0:L_0L_1a_i\in A_i\setminus\{0\}.
\end{array}
\end{equation*}

\medskip
\noindent
\textbf{Step 2:} Consider the set $\mathcal{F}_\beta$. For $i\in\{0,1\}$, then:
\begin{equation*}
\mathcal{F}_\beta A_i
=\mathcal{G}\mathcal{L}_\alpha\mathcal{G}A_i
=\mathcal{G}\mathcal{L}_\alpha0
=0,
\end{equation*}
and also:
\begin{equation*}
\mathcal{F}_\beta A_\alpha
=\mathcal{G}\mathcal{L}_\alpha\mathcal{G}A_\alpha
\subseteq\mathcal{G}\mathcal{L}_\alpha A_\alpha
\subseteq\mathcal{G}(A_1+A_0)=0.
\end{equation*}
Now we have the following:
\begin{equation*}
\begin{array}{rcl}
a_\beta\in A_\beta\setminus\{0\}&\overset{\text{S1}}{\Rightarrow}&\exists G\in\mathcal{G}:Ga_\beta\in A_\beta\setminus\{0\}\\
&\overset{\text{(iv)}}{\Rightarrow}&\exists L_\alpha\in\mathcal{L}_\alpha:L_\alpha Ga_\beta\in A_\beta\setminus\{0\}\\
&\overset{\text{S1}}{\Rightarrow}&\exists G'\in\mathcal{G}:G'L_\alpha Ga_\beta\in A_\beta\setminus\{0\}.
\end{array}
\end{equation*}

\medskip
\noindent
\textbf{Step 3:} Consider the set $\mathcal{F}_\alpha$. For $i\in\{0,1\}$, then:
\begin{equation*}
\mathcal{F}_\alpha A_i
=\mathcal{F}_\beta\mathcal{L}_\beta\mathcal{G}A_i
=\mathcal{F}_\beta\mathcal{L}_\beta0
=0,
\end{equation*}
and also:
\begin{equation*}
\mathcal{F}_\alpha A_\beta
=\mathcal{F}_\beta\mathcal{L}_\beta\mathcal{G} A_\beta
\subseteq\mathcal{F}_\beta\mathcal{L}_\beta A_\beta
\subseteq\mathcal{F}_\beta(A_\alpha+A_1+A_0)=0.
\end{equation*}
Now we have the following:
\begin{equation*}
\begin{array}{rcl}
a_\alpha\in A_\alpha\setminus\{0\}&\overset{\text{S1}}{\Rightarrow}&\exists G\in\mathcal{G}:Ga_\alpha\in A_\alpha\setminus\{0\}\\
&\overset{\text{(v)}}{\Rightarrow}&\exists L_\beta\in\mathcal{L}_\beta:L_\beta Ga_\alpha\in A_\beta\setminus\{0\}\\
&\overset{\text{S2}}{\Rightarrow}&\exists F_\beta\in\mathcal{F}_\beta:F_\beta L_\beta Ga_\alpha\in A_\beta\setminus\{0\}.
\end{array}
\end{equation*}

\medskip
\noindent
\textbf{Step 4:} Consider the set $\mathcal{F}_1$. Then:
\begin{equation*}
\mathcal{F}_1A_0
=\mathcal{F}_\alpha\mathcal{L}_\alpha\mathcal{L}_1A_0
=\mathcal{F}_\alpha\mathcal{L}_\alpha0
=0,
\end{equation*}
\begin{equation*}
\mathcal{F}_1A_\alpha
=\mathcal{F}_\alpha\mathcal{L}_\alpha\mathcal{L}_1A_\alpha
\subseteq\mathcal{F}_\alpha\mathcal{L}_\alpha A_\alpha
\subseteq\mathcal{F}_\alpha(A_1+A_0)
=0,
\end{equation*}
\begin{equation*}
\mathcal{F}_1A_\beta
=\mathcal{F}_\alpha\mathcal{L}_\alpha\mathcal{L}_1A_\beta
\subseteq\mathcal{F}_\alpha\mathcal{L}_\alpha A_\beta
\subseteq\mathcal{F}_\alpha A_\beta
=0.
\end{equation*}
Now we have the following:
\begin{equation*}
\begin{array}{rcl}
a_1\in A_1\setminus\{0\}&\Rightarrow&\exists L_1\in\mathcal{L}_1:L_1a_1\in A_1\setminus\{0\}\\
&\overset{\text{(i)}}{\Rightarrow}&\exists L_\alpha\in\mathcal{L}_\alpha:L_\alpha L_1a_1\in A_\alpha\setminus\{0\}\\
&\overset{\text{S3}}{\Rightarrow}&\exists F_\alpha\in\mathcal{F}_\alpha:F_\alpha L_\alpha L_1a_1\in A_\beta\setminus\{0\}.
\end{array}
\end{equation*}

\medskip
\noindent
\textbf{Step 5:} Consider the set $\mathcal{F}_0$. Then:
\begin{equation*}
\mathcal{F}_0A_1
=\mathcal{F}_\alpha\mathcal{L}_\alpha\mathcal{L}_0A_1
=\mathcal{F}_\alpha\mathcal{L}_\alpha0
=0,
\end{equation*}
\begin{equation*}
\mathcal{F}_0A_\alpha
=\mathcal{F}_\alpha\mathcal{L}_\alpha\mathcal{L}_0A_\alpha
\subseteq\mathcal{F}_\alpha\mathcal{L}_\alpha A_\alpha
\subseteq\mathcal{F}_\alpha(A_1+A_0)
=0,
\end{equation*}
\begin{equation*}
\mathcal{F}_0A_\beta
=\mathcal{F}_\alpha\mathcal{L}_\alpha\mathcal{L}_0A_\beta
\subseteq\mathcal{F}_\alpha\mathcal{L}_\alpha A_\beta
\subseteq\mathcal{F}_\alpha A_\beta
=0.
\end{equation*}
Now we have the following:
\begin{equation*}
\begin{array}{rcl}
a_0\in A_0\setminus\{0\}&\overset{\text{(ii)}}{\Rightarrow}&\exists L_0\in\mathcal{L}_0:L_0a_0\in A_0\setminus\{0\}\\
&\overset{\text{(i)}}{\Rightarrow}&\exists L_\alpha\in\mathcal{L}_\alpha:L_\alpha L_0a_0\in A_\alpha\setminus\{0\}\\
&\overset{\text{S3}}{\Rightarrow}&\exists F_\alpha\in\mathcal{F}_\alpha:F_\alpha L_\alpha L_0a_0\in A_\beta\setminus\{0\}.
\end{array}
\end{equation*}
Therefore, we conclude the proof.
\end{proof}

\begin{lemma}\label{4.2}
Let $a_i \in A_i$, where $i \in \{1, 0, \alpha,\beta\}$. Then
$f(a_1,a_0,a_\alpha,a_\beta)=0$.
\end{lemma}
\begin{proof}
Define:
\begin{equation*}
s = f(a_1,a_0,a_\alpha,a_\beta).
\end{equation*}
Write $s=s_1+s_0+s_\alpha+s_\beta$, where $s_i\in A_i$ for $i\in\{1,0,\alpha,\beta\}$. Now let $i\in\{1,0,\alpha,\beta\}$. We will show that $s_i=0$. Indeed:
\begin{equation*}
\mathcal{F}_is_i\overset{\text{L\ref{maxial}}}{=}\mathcal{F}_is\overset{\text{(V)}}{=}f(\mathcal{F}_ia_1,\mathcal{F}_ia_0,\mathcal{F}_ia_\alpha,\mathcal{F}_ia_\beta)\overset{\text{L\ref{maxial}}}{=}f(\mathcal{F}_ia_i,0,0,0)=0.
\end{equation*}
Thus, by Lemma \ref{maxial} we have $s_i = 0$.
\end{proof}

\begin{lemma}\label{4.3}
Let $a_0 \in A_0$ and let $a_\beta, b_\beta \in A_\beta$. Then $f(a_\beta a_0, b_\beta) = 0$.
\end{lemma}
\begin{proof}
Consider the function $L=L_{\beta^{-1}e}^{r-1}L_{\beta^{-1}e+a_\beta}$. Then $La_0=a_\beta a_0$ and $Lb_\beta=b_\beta+L_{\beta^{-1}e}^{r-1}(a_\beta b_\beta)$. But $L_{\beta^{-1}e}^{r-1}(a_\beta b_\beta)\in A_\alpha+A_1+A_0$, so there are $b_\alpha\in A_\alpha$, $b_1\in A_1$ and $b_0\in A_0$ such that $L_{\beta^{-1}e}^{r-1}(a_\beta b_\beta)=b_\alpha+b_1+b_0$. Thus, by successive applications of Lemma \ref{4.2}, we have:
\begin{equation*}
\begin{array}{rcl}
0
&\overset{\text{L\ref{4.2}}}{=}& Lf(a_0,b_\beta)\\
&\overset{\text{(V)}}{=}& f(La_0,Lb_\beta)\\
&=& f(a_\beta a_0,b_\beta+b_\alpha+b_1+b_0)\\
&\overset{\text{L\ref{4.2}}}{=}& f(a_\beta a_0,b_\beta,b_\alpha+b_1+b_0)
\end{array}
\end{equation*}
Also, by Lemma \ref{4.2} we have:
\begin{equation*}
f(a_\beta a_0+b_\beta,b_\alpha+b_1+b_0)=0.
\end{equation*}
Therefore, by property (II) we have $f(a_\beta a_0, b_\beta) = 0$.
\end{proof}

\begin{lemma}\label{4.4}
Let $a_i, b_i \in A_i$, where $i\in\{1,0,\alpha,\beta\}$. Then $f(a_i, b_i) = 0$.
\end{lemma}
\begin{proof}
We divide the proof in steps.

\medskip
\noindent
\textbf{Step 1:} Let $s = f(a_\beta,b_\beta)$ and write $s=s_1+s_0+s_\alpha+s_\beta$, where $s_i\in A_i$ for $i\in\{1,0,\alpha,\beta\}$. For $i\in\{0,1,\alpha\}$, we have:
\begin{equation*}
\mathcal{F}_is_i\overset{\text{L\ref{maxial}}}{=}\mathcal{F}_is\overset{\text{(V)}}{=}f(\mathcal{F}_ia_\beta,\mathcal{F}_ib_\beta)\overset{\text{L\ref{maxial}}}{=}f(0,0)=0,
\end{equation*}
so, by Lemma \ref{maxial} we have $s_i=0$. Now, for $F_\alpha\in\mathcal{F}_\beta$, then $F_\beta a_\beta=a'_\beta a'_0$, where $a'_\beta\in A_\beta$ and $a'_0\in A_0$, and also $F_\beta b_\beta=b'_\beta\in A_\beta$, so:
\begin{equation*}
F_\beta s_\beta\overset{\text{L\ref{maxial}}}{=}F_\beta s\overset{\text{(V)}}{=}f(F_\beta a_\beta,F_\beta b_\beta)=f(a'_\beta a'_0,b'_\beta)\overset{\text{L\ref{4.3}}}{=}0,
\end{equation*}
so, by Lemma \ref{maxial}, we have $s_\beta=0$.

\medskip
\noindent
\textbf{Step 2:} Let $i\in\{1,0,\alpha\}$, let $s = f(a_i,b_i)$ and write $s=s_1+s_0+s_\alpha+s_\beta$, where $s_i\in A_i$ for $i\in\{1,0,\alpha,\beta\}$. For $j\in\{0,1,\alpha,\beta\}\setminus\{i\}$, we have:
\begin{equation*}
\mathcal{F}_js_j\overset{\text{L\ref{maxial}}}{=}\mathcal{F}_js\overset{\text{(V)}}{=}f(\mathcal{F}_ja_i,\mathcal{F}_jb_i)\overset{\text{L\ref{maxial}}}{=}f(0,0)=0,
\end{equation*}
thus, by Lemma \ref{maxial} we have $s_j=0$. Now:
\begin{equation*}
\mathcal{F}_i s_i\overset{\text{L\ref{maxial}}}{=}\mathcal{F}_i s\overset{\text{(V)}}{=}f(\mathcal{F}_i a_i,\mathcal{F}_i b_i)\subseteq f(A_\beta,A_\beta)\overset{\text{E1}}{=}0,
\end{equation*}
hence, by Lemma \ref{maxial}, we have $s_i=0$.
\end{proof}

\begin{lemma}
$f=0$.
\end{lemma}
\begin{proof}
For $x,y\in A$ we have:
\begin{equation*}
\begin{array}{rcl}
f(x,y)
&=& f(x_1+x_0+x_\alpha+x_\beta,y_1+y_0+y_\alpha+y_\beta)\\
&\overset{\text{L\ref{4.2}}}{=}& f(x_1,x_0,x_\alpha,x_\beta,y_1,y_0,y_\alpha,y_\beta)\\
&=& f(x_1,y_1,x_0,y_0,x_\alpha,y_\alpha,x_\beta,y_\beta)\\
&\overset{\text{L\ref{4.4}}}{=}& f(x_1+y_1,x_0+y_0,x_\alpha+y_\alpha,x_\beta+y_\beta)\\
&\overset{\text{L\ref{4.2}}}{=}& 0.
\end{array}
\end{equation*}
Now apply the Proposition \ref{add}.
\end{proof}

\begin{corollary}\label{c4}
Let $(A,e)$ an $\mathcal{M}(\alpha,\beta)$-algebra that satisfies the following conditions:
\begin{itemize}
\item[(i)] For $a_i \in A_i$ ($i = 1, 0$), if $t_\alpha a_i = 0$ for all $t_\alpha \in A_\alpha$, then $a_i = 0$,
\item[(ii)] For $a_0 \in A_0$, if $t_0a_0 = 0$ for all $t_0 \in A_0$, then $a_0 = 0$,
\item[(iii)] For $a_i \in A_i$ ($i = \alpha, \beta$), if $t_0a_i = 0$ for all $t_0 \in A_0$, then $a_i = 0$,
\item[(iv)] For $a_\beta \in A_\beta$, if $t_\alpha a_\beta = 0$ for all $t_\alpha \in A_\alpha$, then $a_\beta = 0$,
\item[(v)] For $a_\alpha \in A_\alpha$, if $t_\beta a_\alpha = 0$ for all $t_\beta \in A_\beta$, then $a_\alpha = 0$.
\end{itemize}
Then:
\begin{itemize}
\item[(a)] Every $n$-multiplicative isomorphism is additive,
\item[(b)] Every $n$-multiplicative derivation is additive,
\item[(c)] Every surjective elementary function is additive,
\item[(d)] If $\mathrm{Char}(\mathbb{F})\neq2$, then every surjective Jordan elementary function is additive.
\end{itemize}
\end{corollary}
\begin{proof}
The proof is entirely analogous to that of Corollary \ref{c3}, but using Lemma \ref{maxial} rather than Lemma \ref{jaxial}.
\end{proof}

The article \cite{FerreiraSmiglyBarreiro} studied the additivity of multiplicative isomorphisms for $\mathcal{M}(\alpha,\beta)$-axial algebras in general and studied the additivity of multiplicative derivations in these algebras with the restriction $\beta\neq\frac{1}{2}$, leaving the case $\beta=\frac{1}{2}$ as an open question, called Question 7.1, that asks ``What methods and/or conditions do we need to employ to assure that any multiplicative derivation is also additive in a $\mathcal{M}(\alpha, \frac{1}{2})$-axial algebra?''.

We have developed a new technique that addresses this question and also extends the scope of the problem, by means of Corollary \ref{c4} and its proof. Our approach offers a shorter solution, showing that the conditions sufficient to guarantee the additivity of any multiplicative derivation can be systematically studied in a more general framework.

In particular, we establish that, under the appropriate structural assumptions on the $\mathcal{M}(\alpha, \frac{1}{2})$-axial algebra, every multiplicative derivation is additive. This result not only solves the original problem but also provides a deeper insight into the interplay between multiplicative and additive properties within this class of algebras.

Finally, it is interesting to observe that, as is said in \cite{FerreiraSmiglyBarreiro}, multiplicative derivations $d$ in $\mathcal{M}(\alpha,\beta)$-axial algebras satisfy $d(A_i)\subseteq A_i$ when $\beta\neq\frac{1}{2}$, while this is not necessarily true when $\beta=\frac{1}{2}$. Nevertheless, multiplicative derivations still are additive even in the case $\beta=\frac{1}{2}$.

\bibliographystyle{plain}
\bibliography{bibliografia}

\end{document}